\documentclass[a4paper,british]{scrartcl}
\usepackage[T1]{fontenc}
\usepackage[latin9]{inputenc}
\usepackage{babel}
\usepackage{prettyref}
\usepackage{mathtools}
\usepackage{amsmath}
\usepackage{amsthm}
\usepackage{amssymb}
\usepackage{esint}
\usepackage[unicode=true,pdfusetitle,
 bookmarks=true,bookmarksnumbered=false,bookmarksopen=false,
 breaklinks=false,pdfborder={0 0 1},backref=false,colorlinks=false]
 {hyperref}

\makeatletter

\pdfpageheight\paperheight
\pdfpagewidth\paperwidth

\theoremstyle{definition}
\newtheorem*{defn*}{\protect\definitionname}
\theoremstyle{plain}
\newtheorem{thm}{\protect\theoremname}[section]
\theoremstyle{remark}
\newtheorem{rem}[thm]{\protect\remarkname}
\theoremstyle{plain}
\newtheorem{prop}[thm]{\protect\propositionname}
\theoremstyle{plain}
\newtheorem{lem}[thm]{\protect\lemmaname}
\theoremstyle{definition}
\newtheorem{example}[thm]{\protect\examplename}

\@ifundefined{date}{}{\date{}}
\usepackage{prettyref}

\usepackage{enumerate}

\newcommand*{\e}{\mathrm{e}}
\renewcommand*{\i}{\mathrm{i}}
\newcommand{\m}{\operatorname{m}}

\newcommand{\R}{\mathbb{R}}
\newcommand{\N}{\mathbb{N}}
\renewcommand{\C}{\mathbb{C}}
\newcommand{\dom}{\operatorname{dom}}
\newcommand{\ran}{\operatorname{ran}}
\newcommand{\spt}{\operatorname{spt}}
\renewcommand{\d}{\,\mathrm{d}}
\renewcommand{\Re}{\operatorname{Re}}

\newcommand{\IV}{\operatorname{IV}}
\newcommand{\ind}{\operatorname{ind}}

\newrefformat{prop}{Proposition \ref{#1}}
\newrefformat{lem}{Lemma \ref{#1}}
\newrefformat{thm}{Theorem \ref{#1}}
\newrefformat{cor}{Corollary \ref{#1}}
\newrefformat{rem}{Remark \ref{#1}}
\newrefformat{exa}{Example \ref{#1}}
\newrefformat{sec}{Section \ref{#1}}
\newrefformat{sub}{Subsection \ref{#1}}
\newrefformat{eq}{(\ref{#1})}

\theoremstyle{definition}
\newtheorem{hyp}{Hypotheses}

\newrefformat{hyp}{Hypotheses \ref{#1}}

\makeatother

\providecommand{\definitionname}{Definition}
\providecommand{\examplename}{Example}
\providecommand{\lemmaname}{Lemma}
\providecommand{\propositionname}{Proposition}
\providecommand{\remarkname}{Remark}
\providecommand{\theoremname}{Theorem}

\begin{document}
\title{Semigroups associated with differential-algebraic equations}
\author{Sascha Trostorff\thanks{Mathematisches Seminar, CAU Kiel, Germany, email: trostorff@math.uni-kiel.de}}

\maketitle
\textbf{Abstract. }We consider differential-algebraic equations in
infinite dimensional state spaces, and study under which conditions
we can associate a $C_{0}$-semigroup with such equations. We determine
the right space of initial values and characterise the existence of
a $C_{0}$-semigroup in the case of operator pencils with polynomially
bounded resolvents. \medskip{}

\textbf{Keywords: }differential-algebraic equations, $C_{0}$-semigroups,
Wong sequence, operator pencil, polynomially bounded resolvent\textbf{\medskip{}
}

\textbf{2010 MSC: }46N20, 34A09, 35M31\textbf{\medskip{}
}

\section{Introduction}

In the case of matrices the study of differential-algebraic equations
(DAEs), i.e. equations of the form 
\begin{align*}
(Eu)'(t)+Au(t) & =f(t),\\
u(0) & =u_{0}
\end{align*}
for matrices $E,A\in\R^{n\times n}$, is a very active field in mathematics
(see e.g. \cite{Ilchmann2013,Ilchmann2015,Ilchmann2015_1,Ilchmann4_2017}
and the references therein). The main difference to classical differential
equations is that the matrix $E$ is allowed to have a nontrivial
kernel. Thus, one cannot expect to solve the equation for each right
hand side $f$ and each initial value $u_{0}$. In case of matrices
one can use normal forms (see e.g. \cite[Theorem 2.7]{Berger2012,Mehrmann2006})
to determine the `right' space of initial values, so-called \emph{consistent
initial values}. However, this approach cannot be used in case of
operators on infinite dimensional spaces. Another approach uses so-called
Wong sequences associated with matrices $E$ and $A$ (see e.g. \cite{Berger_Trenn2012})
and this turns out to be applicable also in the operator case.\\
In contrast to the finite dimensional case, the case of infinite dimensions
is not that well studied. It is the aim of this article, to generalise
some of the results in the finite dimensional case to infinite dimensions.
For simplicity, we restrict ourselves to homogeneous problems. More
precisely, we consider equations of the form 
\begin{align}
(Eu)'(t)+Au(t) & =0\quad(t>0),\label{eq:DAE0}\\
u(0) & =u_{0},\nonumber 
\end{align}
where $E\in L(X;Y)$ for some Banach spaces $X,Y$ and $A:\dom(A)\subseteq X\to Y$
is a densely defined closed linear operator. We will define the notion
of mild and classical solutions for such equations and determine the
`right' space of initial conditions for which a mild solution could
be obtained.

For doing so, we start with the definition of Wong sequences associated
with $(E,A)$ in \prettyref{sec:Wong-sequence} which turns out to
yield the right spaces for initial conditions. In \prettyref{sec:Necessary-conditions-for}
we consider the space of consistent initial values and provide some
necessary conditions for the existence of a $C_{0}$-semigroup associated
with the above problem under the assumption that the space of consistent
initial values is closed and the mild solutions are unique (\prettyref{hyp:C_0-sg}).
In \prettyref{sec:Pencils-with-polynomial} we consider operators
$(E,A)$ such that $(zE+A)$ is boundedly invertible on a right half
plane and the inverse is polynomially bounded on that half plane.
In this case it is possible to determine the space of consistent initial
values in terms of the Wong sequence and we can characterise the conditions
for the existence of a $C_{0}$-semigroup yielding the mild solutions
of \prettyref{eq:DAE0} at least in the case of Hilbert spaces. One
tool needed in the proof are the Fourier-Laplace transform and the
Theorem of Paley-Wiener, which will be recalled in \prettyref{sec:Preliminaries}.

As indicated above, the study of DAEs in infinite dimensions is not
such an active field of study as for the finite dimensional case.
We mention \cite{Showlater1979} where in Hilbert spaces the case
of selfadjoint operators $E$ is treated using positive definiteness
of the operator pencil. Similar approaches in Hilbert spaces were
used in \cite{Picard_McGhee} for more general equations. However,
in both references the initial condition was formulated as $\left(Eu\right)(0)=u_{0}.$
We also mention the book \cite{Favini1999}, where such equations
are studied with the focus on maximal regularity. Another approach
for dealing with such degenerated equations uses the framework of
set-valued (or multi-valued) operators, see \cite{Favini1993,Knuckels1994}.
Furthermore we refer to \cite{Reis2007,Reis2005}, where sequences
of projectors are used to decouple the system. Moreover, there exist
several references in the Russian literature, where the equations
are called Sobolev type equations (see e.g. \cite{Sviridyuk2003}
an the references therein). Finally, we mention the articles \cite{Thaller1996,Thaller2001},
which are closely related to the present work, but did not consider
the case of operator pencils with polynomially bounded resolvents.
In case of bounded operators $E$ and $A$, equations of the form
\prettyref{eq:DAE0} were studied by the author in \cite{Trostorff_DAE2019,Trostorff_DAE_higherindex_2018},
where the concept of Wong sequences associated with $(E,A)$ was already
used.

We assume that the reader is familiar with functional analysis and
in particular with the theory of $C_{0}$-semigroups and refer to
the monographs \cite{Pazy1983,engel2000one,Yosida}. Throughout, if
not announced differently, $X$ and $Y$ are Banach spaces.

\section{Preliminaries\label{sec:Preliminaries}}

We collect some basic knowledge on the so-called Fourier-Laplace transformation
and weak derivatives in exponentially weighted $L_{2}$-spaces, which
is needed in \prettyref{sec:Pencils-with-polynomial}. We remark that
these concepts were successfully used to study a broad class of partial
differential equations (see e.g. \cite{Picard,Picard_McGhee,Picard2014_survey}
and the references therein).
\begin{defn*}
Let $\rho\in\R$ and $H$ a Hilbert space. Define 
\[
L_{2,\rho}(\R;H)\coloneqq\{f:\R\to H\,;\,f\text{ measurable},\int_{\R}\|f(t)\|\e^{-2\rho t}\d t<\infty\}
\]
with the usual identification of functions which are equal almost
everywhere. Moreover, we define the Sobolev space 
\[
H_{\rho}^{1}(\R;H)\coloneqq\{f\in L_{2,\rho}(\R;H)\,;\,f'\in L_{2,\rho}(\R;H)\},
\]
where the derivative is meant in the distributional sense. Finally,
we define $\mathcal{L}_{\rho}$ as the unitary extension of the mapping
\[
C_{c}(\R;H)\subseteq L_{2,\rho}(\R;H)\to L_{2}(\R;H),\quad f\mapsto\left(t\mapsto\frac{1}{\sqrt{2\pi}}\int_{\R}\e^{-(\i t+\rho)s}f(s)\d s\right).
\]
We call $\mathcal{L}_{\rho}$ the \emph{Fourier-Laplace transform.
}Here, $C_{c}(\R;H)$ denotes the space of $H$-valued continuous
functions with compact support.
\end{defn*}
\begin{rem}
It is a direct consequence of Plancherels theorem, that $\mathcal{L}_{\rho}$
becomes unitary.
\end{rem}

The connection of $\mathcal{L}_{\rho}$ and the space $H_{\rho}^{1}(\R;H)$
is explained in the next proposition.
\begin{prop}[{see e.g. \cite[Proposition 1.1.4]{Trostorff_habil}}]
Let $u\in L_{2,\rho}(\R;H)$ for some $\rho\in\R.$ Then $u\in H_{\rho}^{1}(\R;H)$
if and only if $\left(t\mapsto(\i t+\rho)\left(\mathcal{L}_{\rho}u\right)(t)\right)\in L_{2}(\R;H).$
In this case we have 
\[
\left(\mathcal{L}_{\rho}u'\right)(t)=(\i t+\rho)\left(\mathcal{L}_{\rho}u\right)(t)\quad(t\in\R\text{ a.e.}).
\]
\end{prop}

Moreover, we have the following variant of the classical Sobolev embedding
theorem.
\begin{prop}[{Sobolev-embedding theorem, see \cite[Lemma 3.1.59]{Picard_McGhee}
or \cite[Proposition 1.1.8]{Trostorff_habil}}]
\label{prop:Sobolev} Let $u\in H_{\rho}^{1}(\R;H)$ for some $\rho\in\R.$
Then, $u$ has a continuous representer with $\sup_{t\in\R}\|u(t)\|\e^{-\rho t}<\infty.$
\end{prop}

Finally, we need the Theorem of Paley-Wiener allowing to characterise
those $L_{2}$-functions supported on the positive real axis in terms
of their Fourier-Laplace transform.
\begin{thm}[{Paley-Wiener, \cite{Paley_Wiener} or \cite[19.2 Theorem]{rudin1987real}}]
\label{thm:Paley-Wiener} Let $\rho\in\R.$ We define the Hardy space
\[
\mathcal{H}_{2}(\C_{\Re>\rho};H)\coloneqq\left\{ f:\C_{\Re>\rho}\to H\,;\,f\text{ holomorphic},\,\sup_{\mu>\rho}\int_{\R}\|f(\i t+\mu)\|^{2}\d t<\infty\right\} .
\]
Let $u\in L_{2,\rho}(\R;H).$ Then $\spt u\subseteq\R_{\geq0}$ if
and only if 
\[
\left(\i t+\mu\mapsto\left(\mathcal{L}_{\mu}u\right)(t)\right)\in\mathcal{H}_{2}(\C_{\Re>\rho};H).
\]
\end{thm}

\section{Wong sequence\label{sec:Wong-sequence}}

Throughout, let $E\in L(X;Y)$ and $A:\dom(A)\subseteq X\to Y$ densely
defined closed linear.
\begin{defn*}
For $k\in\N$ we define the spaces $\IV_{k}\subseteq X$ recursively
by 
\begin{align*}
\IV_{0} & \coloneqq\dom(A),\\
\IV_{k+1} & \coloneqq A^{-1}[E[\IV_{k}]].
\end{align*}
This sequence of subspaces is called the \emph{Wong sequence associated
with $(E,A)$}.
\end{defn*}
\begin{rem}
We have
\[
\IV_{k+1}\subseteq\IV_{k}\quad(k\in\N).
\]
Indeed, for $k=0$ this follows from $\IV_{1}=A^{-1}[E[\IV_{0}]]\subseteq\dom(A)=\IV_{0}$
and hence, the assertion follows by induction.
\end{rem}

\begin{defn*}
We define 
\[
\rho(E,A)\coloneqq\{z\in\C\,;\,(zE+A)^{-1}\in L(Y;X)\}
\]
the \emph{resolvent set associated with $(E,A)$}.
\end{defn*}
We start with some useful facts on the Wong sequence. The following
result was already given in \cite{Trostorff_DAE_higherindex_2018}
in case of a bounded operator $A$.
\begin{lem}
\label{lem:observations_IV}Let $k\in\N.$ Then 
\[
E(zE+A)^{-1}A\subseteq A(zE+A)^{-1}E
\]
and 
\[
(zE+A)^{-1}E[\IV_{k}]\subseteq\IV_{k+1}
\]
for each $z\in\rho(E,A)$. Moreover, for $x\in\IV_{k}$ we find elements
$x_{1},\ldots,x_{k}\in X,\,x_{k+1}\in\dom(A)$ such that 
\[
(zE+A)^{-1}Ex=\frac{1}{z}x+\sum_{\ell=1}^{k}\frac{1}{z^{\ell+1}}x_{\ell}+\frac{1}{z^{k+1}}(zE+A)^{-1}Ax_{k+1}\quad(z\in\rho(E,A)\setminus\{0\}).
\]
\end{lem}

\begin{proof}
For $x\in\dom(A)$ we compute 
\begin{align*}
E(zE+A)^{-1}Ax & =E\left(x-(zE+A)^{-1}zEx\right)\\
 & =Ex-zE(zE+A)^{-1}Ex\\
 & =Ex-Ex+A(zE+A)^{-1}Ex\\
 & =A(zE+A)^{-1}Ex.
\end{align*}
We prove the second and third claim by induction. Let $k=0$ and $x\in\IV_{0}=\dom(A).$
Then $(zE+A)^{-1}Ex\in\dom(A)$ with 
\[
A(zE+A)^{-1}Ex=E(zE+A)^{-1}Ax\in E[\dom(A)]=E[\IV_{0}]
\]
and thus, $(zE+A)^{-1}Ex\in\IV_{1}.$ Moreover 
\[
(zE+A)^{-1}Ex=\frac{1}{z}(x-(zE+A)^{-1}Ax)
\]
showing the equality with $x_{1}=-x\in\dom(A)$. Assume now that both
assertions hold for $k\in\N$ and let $x\in\IV_{k+1}$. Then $Ax=Ey$
for some $y\in\IV_{k}$ and we infer 
\[
A(zE+A)^{-1}Ex=E(zE+A)^{-1}Ey\in E[\IV_{k+1}]
\]
by induction hypothesis. Hence, $(zE+A)^{-1}Ex\in\IV_{k+2}.$ Moreover,
by assumption we find $y_{1},\ldots,y_{k}\in X$ and $y_{k+1}\in\dom(A)$
such that 
\begin{align*}
\frac{1}{z}y+\sum_{\ell=1}^{k}\frac{1}{z^{\ell+1}}y_{\ell}+\frac{1}{z^{k+1}}(zE+A)^{-1}Ay_{k+1} & =(zE+A)^{-1}Ey\\
 & =(zE+A)^{-1}Ax\\
 & =x-z(zE+A)^{-1}Ex.
\end{align*}
Thus, we obtain the desired formula with $x_{1}\coloneqq-y,x_{j}=-y_{j-1}$
for $j\in\{2,\ldots,k+2\}$.
\end{proof}
\begin{lem}
\label{lem:closure nice}Assume $\rho(E,A)\ne\emptyset.$ Then for
each $k\in\N$ we have that 
\[
A^{-1}\left[E\left[\overline{\IV_{k}}\right]\right]\subseteq\overline{\IV_{k+1}}.
\]
\end{lem}

\begin{proof}
We prove the claim by induction. For $k=0,$ let $x\in\dom(A)$ such
that $Ax=Ey$ for some $y\in\overline{\IV_{0}}.$ Hence, we find a
sequence $(y_{n})_{n\in\N}$ in $\IV_{0}$ with $y_{n}\to y$ and
since \emph{$E$ }is bounded, we derive $Ey_{n}\to Ey=Ax.$ For $z\in\rho(E,A)$
set 
\[
x_{n}\coloneqq(zE+A)^{-1}zEx+(zE+A)^{-1}Ey_{n}.
\]
By \prettyref{lem:observations_IV} we have that $x_{n}\in\IV_{1}$
and 
\begin{align*}
\lim_{n\to\infty}x_{n} & =(zE+A)^{-1}zEx+(zE+A)^{-1}Ey\\
 & =(zE+A)^{-1}zEx+(zE+A)^{-1}Ax\\
 & =x,
\end{align*}
hence $x\in\overline{\IV_{1}}.$

Assume now that the assertion holds for some $k\in\N$ and let $x\in A^{-1}\left[E\left[\overline{\IV_{k+1}}\right]\right]$.
Then clearly $x\in A^{-1}\left[E\left[\overline{\IV_{k}}\right]\right]\subseteq\overline{\IV_{k+1}}$
and hence, we find a sequence $(w_{n})_{n\in\N}$ in $\IV_{k+1}$
with $w_{n}\to x.$ For $z\in\rho(A,E)$ we infer 
\[
(zE+A)^{-1}zEw_{n}\to(zE+A)^{-1}zEx
\]
and by \prettyref{lem:observations_IV} we have $(zE+A)^{-1}zEx\in\overline{\IV_{k+2}}.$
Moreover, we find a sequence $(y_{n})_{n\in\N}$ in $\IV_{k+1}$ with
$Ax=\lim_{n\to\infty}Ey_{n}.$ As above, we set 
\[
x_{n}\coloneqq(zE+A)^{-1}zEx+(zE+A)^{-1}Ey_{n}
\]
and obtain a sequence in $\overline{\IV_{k+2}}$ converging to $x$.
Hence $x\in\overline{\IV_{k+2}}.$
\end{proof}

\section{Necessary conditions for $C_{0}$-semigroups\label{sec:Necessary-conditions-for}}

In this section we focus on the differential-algebraic problem 
\begin{align}
Eu'(t)+Au(t) & =0\quad(t>0)\label{eq:DAE}\\
u(0) & =u_{0},\nonumber 
\end{align}

where again $E\in L(X;Y)$ and $A:\dom(A)\subseteq X\to Y$ is a linear
closed densely defined operator and $u_{0}\in X$. We begin with the
notion of a classical solution and a mild solution of the above problem.
\begin{defn*}
Let $u:\R_{\geq0}\to X$ be continuous.

\begin{enumerate}[(a)]

\item $u$ is called a \emph{classical solution }of \prettyref{eq:DAE},
if $u$ is continuously differentiable on $\R_{\geq0}$, $u(t)\in\dom(A)$
for each $t\geq0$ and \prettyref{eq:DAE} holds.

\item $u$ is called a \emph{mild solution} of \prettyref{eq:DAE},
if $u(0)=u_{0}$ and for all $t>0$ we have $\int_{0}^{t}u(s)\d s\in\dom(A)$
and 
\[
Eu(t)+A\int_{0}^{t}u(s)\d s=Eu_{0}.
\]

\end{enumerate}
\end{defn*}
Obviously, a classical solution of \prettyref{eq:DAE} is also a mild
solution of \prettyref{eq:DAE}. The main question is now to determine
a natural space, where one should seek for (mild) solutions. In particular,
we have to find the initial values. We define the space of such values
by
\[
U\coloneqq\left\{ u_{0}\in X\,;\,\exists u:\R_{\geq0}\to X\text{ mild solution of }\prettyref{eq:DAE}\right\} .
\]
Clearly, $U$ is a subspace of $X$.
\begin{prop}
\label{prop:U_subset_IV}Let $x\in U$ and $u_{x}$ be a mild solution
of \prettyref{eq:DAE} with initial value $x$. Then $u_{x}(t)\in\bigcap_{k\in\N}\overline{\IV_{k}}$
for each $t\geq0$. In particular, $U\subseteq\bigcap_{k\in\N}\overline{\IV_{k}}.$
\end{prop}

\begin{proof}
Let $t\geq0.$ Obviously, we have that $u_{x}(t)\in\overline{\IV_{0}}=\overline{\dom(A)}=X.$
Assume now that we know $u_{x}(t)\in\overline{\IV_{k}}$ for all $t\geq0$.
We then have 
\[
A\int_{t}^{t+h}u_{x}(s)\d s=Eu_{x}(t+h)-Eu_{x}(t)\in E[\overline{\IV_{k}}]\quad(h>0)
\]
and thus, 
\[
\int_{t}^{t+h}u_{x}(s)\d s\in A^{-1}\left[E\left[\overline{\IV_{k}}\right]\right]\subseteq\overline{\IV_{k+1}}
\]
by \prettyref{lem:closure nice}. Hence, 
\[
u_{x}(t)=\lim_{h\to0}\int_{t}^{t+h}u_{x}(s)\d s\in\overline{\IV_{k+1}}.\tag*{\qedhere}
\]
\end{proof}
We state the following hypothesis, which we assume to be valid throughout
the whole section.

\begin{hyp}\label{hyp:C_0-sg}

The space $U$ is closed and for each $u_{0}\in U$ the mild solution
of \prettyref{eq:DAE} is unique.

\end{hyp}

As in the case of Cauchy problems, we can show that we can associate
a $C_{0}$-semigroup with \prettyref{eq:DAE}. The proof follows the
lines of \cite[Theorem 3.1.12]{ABHN_2011}.
\begin{prop}
Denote for $x\in U$ the unique mild solution of \prettyref{eq:DAE}
by $u_{x}.$ Then the mappings 
\[
T(t):U\to X,\quad x\mapsto u_{x}(t)
\]
for $t\geq0$ define a $C_{0}$-semigroup on $U$. In particular,
$\ran T(t)\subseteq U$ for each $t\geq0$.
\end{prop}

\begin{proof}
Consider the mapping 
\[
\Phi:U\to C(\R_{\geq0};X),\quad x\mapsto u_{x}.
\]
We equip $C(\R_{\geq0};X)$ with the topology induced by the seminorms
\[
p_{n}(f)\coloneqq\sup_{t\in[0,n]}\|f(t)\|\quad(n\in\N)
\]
for which $C(\R_{\geq0};X)$ becomes a Fréchet space. Then $\Phi$
is linear and closed. Indeed, if $(x_{n})_{n\in\N}$ is a sequence
in $U$ such that $x_{n}\to x$ and $u_{x_{n}}\to u$ as $n\to\infty$
for some $x\in U$ and $u\in C(\R_{\geq0},X)$ we derive $\int_{0}^{t}u_{x_{n}}(s)\d s\to\int_{0}^{t}u(s)\d s$
for each $t\geq0$ since $u_{x_{n}}\to u$ uniformly on $[0,t].$
Moreover,
\[
A\int_{0}^{t}u_{x_{n}}(s)\d s=Ex_{n}-Eu_{x_{n}}(t)\to Ex-Eu(t)\quad(n\to\infty)
\]
for each $t\geq0$ and hence, $\int_{0}^{t}u(s)\d s\in\dom(A)$ with
\[
A\int_{0}^{t}u(s)\d s=Ex-Eu(t)\quad(t\geq0).
\]
Finally, since $u(0)=\lim_{n\to\infty}u_{x_{n}}(0)=x,$ we infer that
$u=u_{x}$ and hence, $\Phi$ is closed. By the closed graph theorem
(see e.g. \cite[III, Theorem 2.3]{Schaefer1971}), we derive that
$\Phi$ is continuous. In particular, for each $t\geq0$ the operator
\[
T(t)x=u_{x}(t)=\Phi(x)(t)
\]
is bounded and linear. Moreover, $T(t)x=u_{x}(t)\to x$ as $t\to0$
for each $x\in U.$ We are left to show that $\ran T(t)\subseteq U$
and that $T$ satisfies the semigroup law. For doing so, let $x\in U$
and $t\geq0.$ We define the function $u:\R_{\geq0}\to X$ by $u(s)\coloneqq u_{x}(t+s)=T(t+s)x.$
Then clearly, $u$ is continuous with $u(0)=u_{x}(t)=T(t)x$ and 
\[
\int_{0}^{s}u(r)\d r=\int_{0}^{s}u_{x}(t+r)\d r=\int_{t}^{s+t}u_{x}(r)\d r=\int_{0}^{s+t}u_{x}(r)\d r-\int_{0}^{t}u_{x}(r)\d r\in\dom(A)
\]
for each $s\geq0$ with 
\begin{align*}
A\int_{0}^{s}u(r)\d r & =A\int_{0}^{s+t}u_{x}(r)\d r-A\int_{0}^{t}u_{x}(r)\d r\\
 & =Eu_{x}(s+t)-Eu_{x}(t)\\
 & =Eu(s)-Eu_{x}(t)\quad(s\geq0).
\end{align*}
Hence, $u$ is a mild solution of \prettyref{eq:DAE} with initial
value $u_{x}(t)$ and thus, $u_{x}(t)\in U$. This proves $\ran T(t)\subseteq U$
and 
\[
T(t+s)x=u(s)=T(s)u_{x}(t)=T(s)T(t)x\quad(s,t\geq0,x\in U).\tag*{\qedhere}
\]
\end{proof}
We want to inspect the generator of $T$ a bit closer.
\begin{prop}
\label{prop:generator}Let $B$ denote the generator of the $C_{0}$-semigroup
$T$. Then we have $-EB\subseteq A.$
\end{prop}

\begin{proof}
Let $x\in\dom(B)$. Consequently, $u_{x}\in C^{1}(\R_{\geq0};X)$
and thus, 
\[
A\frac{1}{h}\int_{t}^{t+h}u_{x}(s)\d s=-\frac{1}{h}E\left(u_{x}(t+h)-u_{x}(t)\right)\to-Eu_{x}'(t)\quad(h\to0)
\]
for each $t\geq0$. Since $\frac{1}{h}\intop_{t}^{t+h}u_{x}(s)\d s\to u_{x}(t)$
as $h\to0$ we infer that $u_{x}(t)\in\dom(A)$ for each $t\geq0$
and 
\[
Eu_{x}'(t)+Au_{x}(t)=0\quad(t\geq0),
\]

i.e. $u$ is a classical solution of \prettyref{eq:DAE}. Choosing
$t=0,$ we infer $x\in\dom(A)$ and $EBx=-Ax.$
\end{proof}

\section{Pencils with polynomially bounded resolvent\label{sec:Pencils-with-polynomial}}

Let $E\in L(X;Y)$ and $A:\dom(A)\subseteq X\to Y$ densely defined
closed and linear. Throughout this section we assume the following.

\begin{hyp} \label{hyp:poly_bd}There exist $\rho_{0}\in\R,\,C\geq0$
and $k\in\N$ such that:

\begin{enumerate}[(a)]

\item $\C_{\Re\geq\rho_{0}}\subseteq\rho(E,A)$,

\item $\forall z\in\C_{\Re\geq\rho_{0}}:\:\|(zE+A)^{-1}\|\leq C|z|^{k}.$

\end{enumerate}

\end{hyp}
\begin{defn*}
We call the minimal $k\in\N$ such that there exists $C\geq0$ with
\[
\|(zE+A)^{-1}\|\leq C|z|^{k}\quad(z\in\C_{\Re\geq\rho_{0}})
\]
the \emph{index of $(E,A)$, }denoted by $\ind(E,A).$
\end{defn*}
\begin{prop}
\label{prop:IV_terminates}Consider the Wong sequence $(\IV_{k})_{k\in\N}$
associated with $(E,A).$ Then 
\[
\overline{\IV_{k}}=\overline{\IV_{k+1}}
\]
for all $k>\ind(E,A).$
\end{prop}

\begin{proof}
Since we clearly have $\overline{\IV_{k+1}}\subseteq\overline{\IV_{k}}$
it suffices to prove $\IV_{k}\subseteq\overline{\IV_{k+1}}$ for $k>\ind(E,A).$
So, let $x\in\IV_{k}$ for some $k>\ind(E,A).$ By \prettyref{lem:observations_IV}
there exist $x_{1},\ldots,x_{k}\in X,\,x_{k+1}\in\dom(A)$ such that
\[
(zE+A)^{-1}Ex=\frac{1}{z}x+\sum_{\ell=1}^{k}\frac{1}{z^{\ell+1}}x_{\ell}+\frac{1}{z^{k+1}}(zE+A)^{-1}Ax_{k+1}.\quad(z\in\rho(E,A)).
\]
We define $x_{n}\coloneqq(nE+A)^{-1}nEx$ for $n\in\N_{\geq\rho_{0}}$.
Then $x_{n}\in\IV_{k+1}$ by \prettyref{lem:observations_IV} and
by what we have above 
\[
x_{n}=x+\sum_{\ell=1}^{k}\frac{1}{n^{\ell}}x_{\ell}+\frac{1}{n^{k}}(nE+A)^{-1}Ax_{k+1}\quad(n\in\N_{\geq\rho_{0}}).
\]
Since $k>\ind(E,A),$ we have that $\frac{1}{n^{k}}(nE+A)^{-1}\to0$
as $n\to\infty$ and hence, $x_{n}\to x$ as $n\to\infty,$ which
shows the claim.
\end{proof}
Our next goal is to determine the space $U$. For doing so, we restrict
ourselves to Hilbert spaces $X$.
\begin{prop}
\label{prop:right_U}Assume \prettyref{hyp:C_0-sg} and let $X$ be
a Hilbert space. Then $U=\overline{\IV_{\ind(E,A)+1}}.$
\end{prop}

\begin{proof}
By \prettyref{prop:generator} and \prettyref{prop:IV_terminates}
we have that $U\subseteq\overline{\IV_{\ind(E,A)+1}}$. We now prove
that $\IV_{\ind(E,A)+1}\subseteq U,$ which would yield the assertion.
Let $x\in\IV_{\ind(E,A)+1}$ and $\rho>\max\{0,\rho_{0}\}.$ We define
\[
v(z)\coloneqq(zE+A)^{-1}Ex\quad(z\in\C_{\Re\geq\rho})
\]
and show that $v\in\mathcal{H}_{2}(\C_{\Re\geq\rho};X).$ For doing
so, we use \prettyref{lem:observations_IV} to find $x_{1},\ldots,x_{k}\in X,x_{k+1}\in\dom(A)$,
$k\coloneqq\ind(E,A)+1$, such that 
\[
v(z)=(zE+A)^{-1}Ex=\frac{1}{z}x+\sum_{\ell=1}^{k}\frac{1}{z^{\ell+1}}x_{\ell}+\frac{1}{z^{k+1}}(zE+A)^{-1}Ax_{k+1}\quad(z\in\C_{\Re\geq\rho}).
\]
Then we have 
\[
\|v(z)\|\leq\frac{K}{|z|}\quad(z\in\C_{\Re\geq\rho})
\]
for some constant $K\geq0$ and hence, $v\in\mathcal{H}_{2}(\C_{\Re\geq\rho};X),$
since obviously $v$ is holomorphic. Setting 
\[
u\coloneqq\mathcal{L}_{\rho}^{\ast}v(\i\cdot+\rho)
\]
we thus have $u\in L_{2,\rho}(\R_{\geq0};X)$ by the Theorem of Paley-Wiener,
\prettyref{thm:Paley-Wiener}. Moreover, 
\[
zv(z)-x=\sum_{\ell=1}^{k}\frac{1}{z^{\ell}}x_{\ell}+\frac{1}{z^{k}}(zE+A)^{-1}Ax_{k+1}\quad(z\in\C_{\Re\geq\rho})
\]
 and thus, $z\mapsto zv(z)-x\in\mathcal{H}_{2}(\C_{\Re\geq\rho};X)$
which yields 
\[
(u-\chi_{\R_{\geq0}}x)'=\mathcal{L}_{\rho}^{\ast}\left(\left(\i\cdot+\rho\right)v-x\right)\in L_{2,\rho}(\R_{\geq0};X),
\]
i.e. $u-\chi_{\R_{\geq0}}x\in H_{\rho}^{1}(\R;X),$ which shows that
$u$ is continuous on $\R_{\geq0}$ by the Sobolev embedding theorem,
\prettyref{prop:Sobolev}. We now prove that $u$ is indeed a mild
solution. Since $u-\chi_{\R_{\geq0}}x$ is continuous on $\R,$ we
infer that 
\[
u(0+)-x=0,
\]
and thus $u$ attains the initial value $x$. Moreover, 
\begin{align*}
\left(\mathcal{L}_{\rho}E(u-\chi_{\R_{\geq0}}x)'\right)(t) & =(\i t+\rho)E\left(\mathcal{L}_{\rho}(u-\chi_{\R_{\geq0}}x)\right)(t)\\
 & =(\i t+\rho)Ev(\i t+\rho)-Ex\\
 & =(\i t+\rho)E((\i t+\rho)E+A)^{-1}Ex-Ex\\
 & =-A((\i t+\rho)E+A)^{-1}Ex\\
 & =-A\left(\mathcal{L}_{\rho}u\right)(t)
\end{align*}
for almost every $t\in\R.$ Hence, $u(t)\in\dom(A)$ almost everywhere
and 
\[
-Au(t)=\left(E(u-\chi_{\R_{\geq0}}x)'\right)(t)
\]
for almost every $t\in\R.$ By integrating over an interval $[0,t]$,
we derive 
\[
-\int_{0}^{t}Au(s)\d s=Eu(t)-Ex\quad(t\geq0)
\]
and hence, $u$ is a mild solution of \prettyref{eq:DAE}. Thus, $x\in U$
and so, $U=\overline{\IV_{\ind(E,A)+1}}.$
\end{proof}
For sake of readability, we introduce the following notion.
\begin{defn*}
We define the space 
\[
V\coloneqq A^{-1}\left[E\left[\overline{\IV_{\ind(E,A)+1}}\right]\right].
\]
\end{defn*}
\begin{rem}
\label{rem:space V}Note that $\IV_{\ind(E,A)+2}\subseteq V\subseteq\overline{\IV_{\ind(E,A)+2}}=\overline{\IV_{\ind(E,A)+1}}$
by \prettyref{lem:closure nice} and \prettyref{prop:IV_terminates}.
\end{rem}

\begin{lem}
\label{lem:op_C}Assume that $E:\overline{\IV_{\ind(E,A)+1}}\to Y$
is injective. Then 
\[
C\coloneqq E^{-1}A:V\subseteq\overline{\IV_{\ind(E,A)+1}}\to\overline{\IV_{\ind(E,A)+1}}
\]
is well-defined and closed.
\end{lem}

\begin{proof}
Note that $A[V]\subseteq E[\overline{\IV_{\ind(E,A)+1}}]$ and thus,
$C$ is well-defined. Let $(x_{n})_{n\in\N}$ by a sequence in $V$
such that $x_{n}\to x$ and $Cx_{n}\to y$ in $\overline{\IV_{\ind(E,A)+1}}$
for some $x,y\in\overline{\IV_{\ind(E,A)+1}}.$ We then have 
\[
Ax_{n}=ECx_{n}\to Ey
\]
and hence, $x\in\dom(A)$ with $Ax=Ey\in E\left[\overline{\IV_{\ind(E,A)+1}}\right].$
This shows, $x\in V$ and $Cx=E^{-1}Ax=y$, thus $C$ is closed.
\end{proof}
\begin{prop}
\label{prop:necessary_gen}Assume \prettyref{hyp:C_0-sg} and let
$X$ be a Hilbert space. Denote by $B$ the generator of $T$. Then
$E:\overline{\IV_{\ind(E,A)+1}}\to Y$ is injective and $B=-C$, where
$C$ is the operator defined in \prettyref{lem:op_C}.
\end{prop}

\begin{proof}
By \prettyref{prop:generator} we have $-EB\subseteq A.$ Hence, for
$x\in U=\overline{\IV_{\ind(E,A)+1}}$ (see \prettyref{prop:right_U})
and $z\in\rho(B)\cap\rho(E,A)$ we obtain 
\[
(zE+A)(z-B)^{-1}x=(zE-EB)(z-B)^{-1}x=Ex
\]
and hence, 
\[
(z-B)^{-1}x=(zE+A)^{-1}Ex.
\]
Thus, if $Ex=0$ for some $x\in\overline{\IV_{\ind(E,A)+1}}$, we
infer that $(z-B)^{-1}x=0$ and thus, $x=0.$ Hence, $E$ is injective
and thus, $C$ is well defined. Moreover, we observe that for $x,y\in\overline{\IV_{\ind(E,A)+1}}$
and $z\in\rho(E,A)\cap\rho(B)$ we have
\begin{align*}
x\in\dom(C)\wedge(z+C)x=y & \Leftrightarrow x\in\dom(A):\left(zE+A\right)x=Ey\\
 & \Leftrightarrow x=(zE+A)^{-1}Ey=(z-B)^{-1}y
\end{align*}
and thus, $z\in\rho(-C)$ with $(z+C)^{-1}=(z-B)^{-1},$ which in
turn implies $B=-C$.
\end{proof}
The converse statement also holds true, even in the case of a Banach
space $X$.
\begin{prop}
\label{prop:sufficient} Let $E:\overline{\IV_{\ind(E,A)+1}}\to Y$
be injective and $-C$ generate a $C_{0}$-semigroup on $\overline{\IV_{\ind(E,A)+1}}$,
where $C$ is the operator defined in \prettyref{lem:op_C}. Then
\prettyref{hyp:C_0-sg} holds.
\end{prop}

\begin{proof}
Denote by $T$ the semigroup generated by $-C.$ By \prettyref{prop:U_subset_IV}
and \prettyref{prop:IV_terminates} we know that $U\subseteq\overline{\IV_{\ind(E,A)+1}}.$
We first prove equality here. For doing so, we need to show that $T(\cdot)x$
is a mild solution of \prettyref{eq:DAE} for $x\in\overline{\IV_{\ind(E,A)+1}}.$
We have 
\[
T(t)x+C\int_{0}^{t}T(s)x\d s=x\quad(t\geq0).
\]
Since $EC\subseteq A$, we know that 
\[
\int_{0}^{t}T(s)x\d s\in\dom(A)\quad(t\geq0)
\]
and that 
\[
A\int_{0}^{t}T(s)x\d s=EC\int_{0}^{t}T(s)x\d s=Ex-ET(t)x
\]
and thus, $T(\cdot)x$ is a mild solution of \prettyref{eq:DAE},
which in turn implies $x\in U.$ So, we indeed have $U=\overline{\IV_{\ind(E,A)+1}}$
and hence, $U$ is closed. It remains to prove the uniqueness of mild
solutions for initial values in $U$. So, let $u_{x}$ be a mild solution
for some $x\in U.$ By \prettyref{prop:U_subset_IV} we know that
$u_{x}(t)\in\overline{\IV_{\ind(E,A)+1}}$ for each $t\geq0$. Hence,
\[
A\int_{0}^{t}u_{x}(s)\d s=Ex-Eu_{x}(t)\in E\left[\overline{\IV_{\ind(E,A)+1}}\right]\quad(t\geq0),
\]
which shows $\int_{0}^{t}u_{x}(s)\d s\in V=\dom(C).$ Hence, 
\[
C\int_{0}^{t}u_{x}(s)\d s=E^{-1}A\int_{0}^{t}u_{x}(s)\d s=x-u_{x}(t)\quad(t\geq0),
\]
i.e. $u_{x}$ is a mild solution of the Cauchy problem associated
with $-C$. Hence, $u_{x}=T(\cdot)x,$ which shows the claim.
\end{proof}
We summarise our findings of this section in the following theorem.
\begin{thm}
We consider the following two statements.

\begin{enumerate}[(a)]

\item \prettyref{hyp:C_0-sg} holds,

\item $E:\overline{\IV_{\ind(E,A)+1}}\to Y$ is injective and $-C$
generates a $C_{0}$-semigroup on $\overline{\IV_{\ind(E,A)+1}}$,
where $C$ is the operator defined in \prettyref{lem:op_C}.

\end{enumerate}

Then (b) $\Rightarrow$ (a) and if $X$ is a Hilbert space, then (b)
$\Leftrightarrow$ (a).
\end{thm}

The crucial condition for \prettyref{hyp:C_0-sg} to hold is the injectivity
of $E:\overline{\IV_{\ind(E,A)+1}}\to Y.$ It is noteworthy that $E|_{\IV_{\ind(E,A)+1}}$
is always injective. Indeed, if $Ex=0$ for some $x\in\IV_{\ind(E,A)+1},$
we can use \prettyref{lem:observations_IV} to find $x_{1},\ldots,x_{\ind(E,A)+1}\in X,x_{\ind(E,A)+2}\in\dom(A)$
such that 
\[
(zE+A)^{-1}Ex=\frac{1}{z}x+\sum_{\ell=1}^{\ind(E,A)+1}\frac{1}{z^{\ell+1}}x_{\ell}+\frac{1}{z^{\ind(E,A)+2}}(zE+A)^{-1}Ax_{\ind(E,A)+2}\quad(z\in\rho(E,A)).
\]
Thus, we have $0=z(zE+A)^{-1}Ex\to x$ as $z\to\infty$ and hence,
$x=0.$ However, it is not true in general,that the injectivity carries
over to the closure $\overline{\IV_{\ind(E,A)+1}}$ as the following
example shows.
\begin{example}
Consider the Hilbert space $L_{2}(-2,2)$ and define the operator
\[
\partial^{\#}:\dom(\partial^{\#})\subseteq L_{2}(-2,2)\to L_{2}(-2,2),\quad u\mapsto u',
\]
where 
\[
\dom(\partial^{\#})\coloneqq\{u\in H^{1}(-2,2)\,;\,u(-2)=u(2)\}.
\]
It is well-known that this operator is skew-selfadjoint. We set 
\[
E\coloneqq\chi_{[-1,1]}(\m),\quad A\coloneqq\chi_{[-2,2]\setminus[-1,1]}(\m)+\partial^{\#},
\]
where $\chi_{I}(\m)$ denotes the multiplication operator with the
function $\chi_{I}$ on $L_{2}(-2,2).$ Clearly, $E$ is linear and
bounded and $A$ is closed linear and densely defined. Moreover, for
$z\in\C_{\Re>0}$ and $u\in\dom(\partial^{\#})$ we obtain 
\begin{align*}
\Re\langle(zE+A)u,u\rangle & =\Re\langle z\chi_{[-1,1]}(\m)u+\chi_{[-2,2]\setminus[-1,1]}(\m)u,u\rangle\\
 & =\Re z\|\chi_{[-1,1]}(\m)u\|_{L_{2}(-2,2)}^{2}+\|\chi_{[-2,2]\setminus[-1,1]}(\m)u\|_{L_{2}(-2,2)}^{2}\\
 & \geq\min\{\Re z,1\}\|u\|_{L_{2}(-2,2)}^{2},
\end{align*}
where we have used the skew-selfadjointness of $\partial^{\#}$ in
the first equality. Hence, we have 
\[
\|u\|_{L_{2}(-2,2)}\leq\frac{1}{\min\{\Re z,1\}}\|(zE+A)u\|_{L_{2}(-2,2)},
\]
which proves the injectivity of $(zE+A)$ and the continuity of its
inverse. Since the same argumentation works for the adjoint $(zE+A)^{\ast}$,
it follows that $(zE+A)^{-1}\in L(L_{2}(-2,2))$ with 
\[
\|(zE+A)^{-1}\|\leq\frac{1}{\min\{\Re z,1\}}\quad(z\in\C_{\Re>0}).
\]
Hence, $(E,A)$ satisfies \prettyref{hyp:poly_bd} on $\C_{\Re\geq\rho_{0}}$
for each $\rho_{0}>0$ with $\ind(E,A)=0.$ Moreover, we have 
\[
u\in\IV_{1}=A^{-1}[E[\dom(A)]]
\]
if and only if $u\in\dom(\partial^{\#})$ and 
\[
\chi_{[-2,2]\setminus[-1,1]}(\m)u+u'=\chi_{[-1,1]}(\m)v
\]
for some $v\in\dom(\partial^{\#}).$ The latter is equivalent to $u\in\dom(\partial^{\#})\cap H^{2}(-1,1)$
and 
\[
u(t)+u'(t)=0\quad(t\notin[-1,1]\text{ a.e}).
\]
Thus, we have 
\begin{align*}
\IV_{1}= & \left\{ u\in\dom(\partial^{\#})\cap H^{2}(-1,1)\,;\:\exists c\in\R:\right.\\
 & \hfill\phantom{aaaaaaaaaaa}\left.\,u(t)=c\left(\chi_{[-2,-1]}(t)\e^{-t}+\chi_{[1,2]}(t)\e^{4-t}\right)\quad(t\notin[-1,1]\text{ a.e.})\right\} .
\end{align*}
In particular, we obtain that 
\[
v(t)\coloneqq\chi_{[-2,-1]}(t)\e^{-t}+\chi_{[1,2]}(t)\e^{4-t}\quad(t\in(-2,2))
\]
belongs to $\overline{\IV_{1}}$. But this function satisfies $Ev=0$
and hence, $E$ is not injective on $\overline{\IV_{1}}.$
\end{example}

\begin{rem}
\label{rem:k=00003D0_nice}In the case $E,A\in L(X;Y)$ and $\ind(E,A)=0,$
the injectivity of $E$ carries over to $\overline{\IV_{1}}.$ Indeed,
we observe that the operators 
\[
(nE+A)^{-1}nE=1-(nE+A)^{-1}
\]
for $n\in\N$ large enough are uniformly bounded. Moreover, for $x\in\IV_{1}$
we have 
\[
(nE+A)^{-1}nEx\to x\quad(n\to\infty)
\]
and hence, the latter converges carries over to $x\in\overline{\IV_{1}}.$
In particular, if $Ex=0$ for some $x\in\overline{\IV_{1}},$ we infer
$x=0$ and thus, \emph{$E$ }is indeed injective on $\overline{\IV_{1}}.$
So far, the author is not able to prove or disprove that the injectivity
also holds for $\ind(E,A)>0$ if $E$ and $A$ are bounded.
\end{rem}

\section*{Acknowledgement}

We thank Felix Schwenninger for pointing our attention to the concept
of Wong sequences in matrix calculus. Moreover, we thank Florian Pannasch
for the observation in \prettyref{rem:k=00003D0_nice} and the anonymous
referee for drawing our attention to the subject of Sobolev type equations
in the Russian literature.

\end{document}